\documentclass[12pt]{article}
\usepackage{amsmath}
\usepackage{amssymb}
\usepackage{latexsym}
\usepackage{amsthm}
\usepackage{mathrsfs}
\usepackage{enumitem}
\usepackage[colorlinks,
            linkcolor=blue,
            anchorcolor=blue,
            citecolor=blue
            ]{hyperref}

\newtheorem{thm}{Theorem}[section]
\newtheorem{lem}[thm]{Lemma}

\newtheorem{cor}[thm]{Corollary}

\def\ZZ{{\mathbb Z}}

\newcommand{\Z}{\mathbb{Z}}

\newcommand{\SW}{\mathcal{SW}}

\newcommand{\asc}{{\rm asc}}

\newcommand{\des}{{\rm des}}

\newcommand{\esd}{{\rm esd}}

\newcommand{\exc}{{\rm exc}}

\newcommand{\fix}{{\rm fix}}
\newcommand{\col}{{\rm col}}
\newcommand{\bad}{{\rm bad}}
\newcommand{\fS}{{\mathfrak S}}

\def\newop#1{\expandafter\def\csname #1\endcsname{\mathop{\rm
			#1}\nolimits}}
		
\newop{rev}

\begin{document}

\begin{center}
{\large \bf  Real-rootedness of variations of Eulerian polynomials}
\end{center}

\begin{center}
James Haglund\textsuperscript{a}  and   Philip B. Zhang\textsuperscript{b} \\[6pt]

$^{a}$Department of Mathematics\\
University of Pennsylvania, Philadelphia, PA 19104-6395, USA\\[6pt]

$^{b}$College of Mathematical Science\\
 Tianjin Normal University, Tianjin  300387, China\\[6pt]

Email:   
$^{a}${\tt  jhaglund@math.upenn.edu},   $^{b}${\tt zhang@tjnu.edu.cn}

\end{center}

\noindent\textbf{Abstract.} 
The binomial Eulerian polynomials, introduced by Postnikov, Reiner, and Williams, are $\gamma$-positive polynomials and can be interpreted as $h$-polynomials of certain flag simplicial polytopes.
Recently,  Athanasiadis studied analogs of these
polynomials for colored permutations.
In this paper, we generalize them to  $\mathbf{s}$-inversion sequences and prove that these new polynomials have only real roots by the method of interlacing polynomials.
Three applications of this result are presented. The first one is to prove the real-rootedness of binomial Eulerian polynomials, which confirms a conjecture of Ma, Ma, and Yeh.
The second one is to prove that the symmetric decomposition of binomial Eulerian polynomials for colored permutations is real-rooted.
Thirdly, our polynomials for certain $\mathbf{s}$-inversion sequences are shown to admit a similar geometric interpretation related to edgewise subdivisions of simplexes.

\noindent  \textbf{AMS Classification 2010:}  05A15, 26C10, 52B05.


\section{Introduction}
The original motivation of this paper is to study the real-rootedness of binomial Eulerian polynomials.
For any positive integer $n$, let   $[n]$  be the set $\{1,2,\ldots,n\}$. Denote by  $\mathfrak{S}_n$  the set of  permutations of $[n]$.  
Given a permutation $\pi = \pi_1\pi_2\ldots\pi_n\in\mathfrak{S}_n$, the descent number  of $\pi$ is the number of   $i\in[n-1]$  satisfying  $\pi_i>\pi_{i+1}$ and the excedance number of $\pi$ is the number of $i\in [n]$ such that $\pi_i>i$.
Recall that $\pi$ is a derangement if $\pi_i\neq i$ for all $i\in [n]$, and denote  by $\mathfrak{D}_n$  the set of  derangements in $\mathfrak{S}_n$.  
The polynomials
\[
A_n(z) :=\sum_{\pi\in\mathfrak{S}_n}z^{\des (\pi)} 
\qquad
\mbox{and}
\qquad
d_n(z):=\sum_{\pi\in\mathfrak{D}_n}z^{\exc (\pi)}
\]
are known as the \emph{Eulerian polynomial} and \emph{derangement polynomial}, respectively.  
A common interesting property of these two polynomials is the  $\gamma$-positivity.
Recall that a polynomial $h(z)$ with nonnegative integer coefficients is said to be $\gamma$-positive, if it admits an expansion of the form 
\[
h(z) = \sum_{i=0}^{\lfloor n/2 \rfloor} 
\gamma_{i} \, z^i (1+z)^{n-2i} 
\]
where $\gamma_{i}$ are  nonnegative integers.
Gamma-positivity directly implies palindromicity and unimodality and  appears widely in combinatorial and geometric contexts, see~\cite{Athanasiadis2017survey} for a survey. 
The following variation of Eulerian polynomials
\begin{equation*}
\widetilde{A}_n(z) \ := 1+ z \sum_{m=1}^n 
{n \choose m}  A_m(z),
\end{equation*}
first studied by Postnikov, Reiner, and Williams~\cite[Section~10.4]{Postnikov2008Faces}, are also $\gamma$-positive and have attracted a lot of interest recently \cite{Shareshian2017Gamma, Ma2017Recurrence, Athanasiadis2018Binomial, Lin2018Sign}.  
Shareshian and Wachs~\cite{Shareshian2017Gamma} called them \emph{binomial Eulerian polynomials} and further studied a  symmetric function generalization of them, which are shown to be  equivariant $\gamma$-positive.  
Another common property of $A_n(z)$ and $d_n(z)$ is that they both have  only real roots, proved by~Frobenius~\cite{Frobenius1910Uber}  and Zhang~\cite{Zhang1995$q$}, respectively.
It is natural to ask whether $\widetilde{A}_n(z)$ is real-rooted as well, which was conjectured by Ma, Ma, and Yeh~\cite{Ma2017Recurrence} based on empirical evidence. 

Eulerian polynomials and derangement polynomials can be generalized to $\mathbf{s}$-inversion sequences.
Given a sequence of positive integers $\mathbf{s} = (s_1,\ldots,s_n)$, define the set of { $\mathbf{s}$-inversion sequences} of length $n$ by
\[
\mathcal{I}_n^{\mathbf{s}} :=
\{
(e_1,\ldots,e_n) \in\Z^n
\,:\,
0\leq e_i<s_i, \, 0\leq i\leq n
\}
\]
with the assumption that $e_0 = e_{n+1} =0$ and $s_0 = s_{n+1} = 1$. 
Following~\cite{Gustafsson2018Box, Gustafsson2018Derangements}, an index $i \in[0,n]$ of  an inversion sequence $\mathbf{e} = (e_1,\ldots, e_n)\in \mathcal{I}_n^\mathbf{s}$ is said to be an  ascent if $\frac{e_i}{s_i}<\frac{e_{i+1}}{s_{i+1}}$, 
a  collision if $\frac{e_i}{s_i}=\frac{e_{i+1}}{s_{i+1}}$, and 
a descent if $\frac{e_i}{s_i}>\frac{e_{i+1}}{s_{i+1}}$,
and denote by $\asc (\mathbf{e})$,  $\col  (\mathbf{e})$, and $\des (\mathbf{e})$  the number of ascents, collisions, and descents in $\mathbf{e}$, respectively.
Let $\mathcal{D}_n^\mathbf{s}$ be the subset of $\mathcal{I}_n^{\mathbf{s}}$ consisting of $\mathbf{e}$ with $\col (\mathbf{e}) =0$.
The {$\mathbf{s}$-Eulerian polynomial}  and {$\mathbf{s}$-derangement polynomial} are defined  as 
\begin{align*}
	E_n^{\mathbf{s}}(z) :=\sum_{\mathbf{e} \in \mathcal{I}_n^{\mathbf{s}}} z^{\asc(\mathbf{e})}
\quad \mbox{ and } \quad
	d_n^\mathbf{s}(z) :=\sum_{\mathbf{e} \in \mathcal{D}_n^\mathbf{s}} z^{\asc(\mathbf{e})},
\end{align*}
respectively.
The real-rootedness of $E_n^{\mathbf{s}}(z)$ and $d_n^\mathbf{s}(z)$ was proved by Savage and Vistonai~\cite{Savage2015$s$}, and Gustafsson and Solus~\cite{Gustafsson2018Derangements}, respectively. Both proofs are via the method of interlacing polynomials, which has also been widely used to prove the real-rootedness of several polynomials arising in  combinatorics (\cite{Jochemko2018real, Yang2014Mutual, Zhang2015real, Leander2016Compatible, Solus2019Simplices}). 

In this paper, we generalize the notion of binomial Eulerian polynomials to $\mathbf{s}$-inversion sequences as follows:
\begin{align}\label{eq:wE}
	\widetilde{E}_n^\mathbf{s}(z) :=\sum_{\mathbf{e} \in \mathcal{I}_n^\mathbf{s}} (1+z)^{\col (\mathbf{e})} z^{\asc(\mathbf{e})}.
\end{align}
The main objective of this paper is  to prove the real-rootedness of $\widetilde{E}_n^\mathbf{s}(z)$. 
To this end, let us first recall some notion about interlacing polynomials.
Given two real-rooted polynomials $f(z)$ and $g(z)$ with positive leading coefficients, let $\{u_i\}$ and $\{v_j\}$ be the set of zeros of $f(z)$ and $g(z)$, respectively.
Recall that \emph{$g(z)$ interlaces $f(z)$}, denoted $g(z)\ll f(z)$, if either
$\deg f(z)=\deg g(z)=d$ and
\begin{align*}
v_d\le u_d \le v_{d-1}\le\cdots\le v_2\le u_2\le v_1\le u_1,
\end{align*}
or $\deg f(z)=\deg g(z)+1=d$ and
\begin{align*}
u_{d}\le v_{d-1}\le\cdots\le v_{2}\le u_{2}\le v_{1}\le u_{1}.
\end{align*}
For convention, we let $0 \ll f$ and $f \ll 0$ for any real-rooted polynomial $f$.
Following Br{\"a}nd{\'e}n~\cite{Braenden2015Unimodality}, a sequence of real polynomials $(f_{1}(z),\dots,f_{m}(z))$
with positive leading coefficients is said to be \emph{interlacing} if $f_{i}(z) \ll f_{j}(z)$ for all $1\le i<j\le m$.
In this paper,  we consider a refinement of $\widetilde{E}_n^\mathbf{s}(z)$, similar to those of  $E_n^{\mathbf{s}}(z)$ and $d_n^{\mathbf{s}}(z)$.
For $1\le m \le n$ and $0\leq k<s_m$, define 
the set  by
\[
\mathcal{J}_m^{\mathbf{s}} :=
\{
 (e_1,\ldots,e_m) \in\Z^m
\,:\,
0\leq e_i<s_i, \, 0\leq i\leq m
\}
\]
with the assumption that $e_0 =0$ and $s_0 = 1$. 
Let $\chi(S)$ be $1$ if $S$ is a true statement and $0$ otherwise.
Now we define the refined polynomials as
\[
p_{m,k}^\mathbf{s}(z) := \sum_{\mathbf{e} = (e_1,\ldots,e_m) \in \mathcal{J}_m^\mathbf{s}}\chi(e_m = k)  (1+z)^{\col' (\mathbf{e})}  z^{\asc(\mathbf{e})},
\]
where $\col' (\mathbf{e}) := |\{ i\in[0,m-1]: \frac{e_i}{s_i}=\frac{e_{i+1}}{s_{i+1}}  \}|$. Note that when $\mathbf{e} = (e_1,\ldots,e_n)$, $\col (\mathbf{e}) = \col' (\mathbf{e}) + \chi(e_n=0)$.
It is clear that
\begin{align*}
		\widetilde{E}_n^\mathbf{s}(z)   = \ (1+z)p_{n,0}^\mathbf{s}(z) +   \sum_{k=1}^{s_n-1}p_{n,k}^\mathbf{s}(z). 
\end{align*}
In this paper, we shall prove the following theorem by the method of interlacing polynomials.
\begin{thm}\label{thm:s}
	Let $\mathbf{s} = (s_1,\ldots, s_n)$ be a sequence of positive integers.  
	Then for any $1\le m \le n$ the  sequence $\left( p_{m,k}^\mathbf{s}(z) \right)_{k=0}^{s_m-1}$ is interlacing and therefore the polynomial $\widetilde{E}_n^\mathbf{s}(z)$ has only real roots.
\end{thm}

The rest of this paper is organized as follows. Section \ref{sect-interlacing} is dedicated to the proof of Theorem~\ref{thm:s}. To prove it, we investigate a new kind of interlacing-preserving matrices with entries $1$, $z$, and $1+z$.
Then three applications are presented in Section~\ref{sect-applications}.
The first one  is the real-rootedness of $\widetilde{A}_n(z)$, which confirms Ma, Ma, and Yeh's conjecture.
Another application is the real-rootedness of  $\widetilde{A}^+_{n,r}(z)$ and  $\widetilde{A}^-_{n,r}(z)$,  the sum of which form the binomial Eulerian polynomials for colored permutations $\widetilde{A}_{n,r}(z)$,  defined by Athanasiadis~\cite{Athanasiadis2018Binomial} recently.
The polynomials $\widetilde{A}_n(z)$ and $\widetilde{A}^+_{n,r}(z)$ can be  interpreted as  $h$-polynomials of  boundary complexes of certain simplicial polytopes, see \cite{Athanasiadis2018Binomial, Postnikov2008Faces}.
In our third application, the  polynomials $\widetilde{E}_n^\mathbf{s}(z)$ for certain $\mathbf{s}$-inversion sequences are shown to be such $h$-polynomials, which are related to  the edgewise subdivisions of simplexes.
An alternative approach to the real-rootedness of these polynomials are also presented.

\section{Interlacing}\label{sect-interlacing}

In this section, we shall  prove Theorem~\ref{thm:s}.
In order to prove the interlacing property of a family of polynomials, it is desirable to prove that  the polynomials of interest satisfy a recursion that produces a new interlacing sequence from an old one.

\begin{lem}
Let $\mathbf{s} = (s_1,s_2,\ldots, s_n)$. 
For $1\le m \le n$ and $0\leq k< s_m$, let $t_{m,k}:=\left\lceil\frac{ks_{m-1}}{s_m}\right\rceil$. 
The  sequence $\left( p_{m,k}^\mathbf{s}(z) \right)_{k=0}^{s_m-1}$   satisfies the following recurrence relation
\begin{align}\label{eqn: recursion}
p_{m,k}^\mathbf{s}(z)  = 	\begin{cases}
z\sum\limits_{i=0}^{t_{m,k}-1}p_{m-1,i}^\mathbf{s}(z)+(1+z) p_{m-1,t_{m,k}}^\mathbf{s}(z) + \sum\limits_{i=t_{m,k}+1}^{s_n-1}p_{m-1,i}^\mathbf{s}(z), & \text{if  }s_m \big| \, ks_{m-1}, \\[10pt]
z\sum\limits_{i=0}^{t_{m,k}-1}p_{m-1,i}^\mathbf{s}(z)+ \sum\limits_{i=t_{m,k}}^{s_m-1}p_{m-1,i}^\mathbf{s}(z), & \text{if  }s_m \not\big| \, ks_{m-1}.
\end{cases}
\end{align}
with the initial conditions 
\[
p_{1,0}^\mathbf{s}(z) = 1+z
\qquad
\mbox{and} 
\qquad
p_{1,k}^\mathbf{s}(z) = z
\qquad
\mbox{ for  \, $1\leq k<s_1$}.
\]
\end{lem}
\begin{proof}
	The initial conditions are easy to check. 
The recursion~\eqref{eqn: recursion} holds, since if $\mathbf{e} = (e_1,\ldots,e_m)\in{\mathcal{J}_m^\mathbf{s}}$ with $e_m = k$ then
\begin{itemize}
	\item $m-1$ is a ascent in $\mathbf{e}$ if and only if $\frac{e_{m-1}}{s_{m-1}}<\frac{k}{s_m}$,  equivalently, $e_{m-1}<t_{m,k}$.
	\item $m-1$ is a collision  in $\mathbf{e}$ if and only if $\frac{e_{m-1}}{s_{m-1}}=\frac{k}{s_m}$, equivalently, $s_m \big| \, ks_{m-1}$.  
\end{itemize}
This completes the proof.
\end{proof}

As usual, it is more convenient to express such recursions via matrix multiplications:
\begin{align}\label{eq-ma-f-g}
  (g_1, \ldots, g_p)^T = G\cdot (f_1, \ldots, f_q)^T,
\end{align}
where $G= (G_{ij}(z))$ is  a  $p \times q$ matrix of polynomials.
A characterization of such matrices was due to Br{\"a}nd{\'e}n~\cite{Braenden2015Unimodality}.

\begin{lem} [{\cite[Theorem 8.5]{Braenden2015Unimodality}}]\label{pb}
	 Let  $\mathcal{F}_q^+$ be the set of  interlacing
	 sequences $(f_i(z))_{i=1}^q$ such that all the coefficients of $f_i(z)$ are  nonnegative  for all $1 \leq i \leq q$. Suppose that  $G= (G_{ij}(z))$  is a $p \times q$ matrix of polynomials. Then $G : \mathcal{F}_q^+ \rightarrow \mathcal{F}_p^+$  if and only if
	\begin{enumerate}
		\item[(1)] $G_{ij}(z)$ has nonnegative coefficients for all $1\le i \le p$ and $1\le j \le q$, and
		\item[(2)] for all $\lambda, \mu >0$, $1\leq i< j \leq q$ and $1\leq k< \ell \leq p$,
		\begin{equation}\label{pbxvw}
		(\lambda z + \mu) G_{kj}(z) + G_{\ell j}(z) \ll (\lambda z + \mu) G_{ki}(z) + G_{\ell i}(z).
		\end{equation}
	\end{enumerate}
\end{lem}

We also need the following lemma.
\begin{lem}[{\cite[Lemma 2.6]{Borcea2010Multivariate}}]\label{lem:convex}
	Let $f$, $g$, and $h$ be real-rooted polynomials  with nonnegative coefficients. 
	\begin{itemize}
		\item  	If $f \ll g$ and $f \ll h$, then $f \ll g+h$.
		\item 	If $f \ll g$ and $h \ll g$, then $f+h \ll g$.
	\end{itemize}
\end{lem}

In this section, we shall prove that the following recursion preserves interlacing, which generalizes~\cite[Lemma~4.4]{Solus2019Simplices}.
\begin{thm}
	\label{thm: interlacing recursion}
    Suppose that $\left( f_i (z)\right)_{i=1}^{q}$ is  a polynomial sequence with nonnegative coefficients.  
   Define another polynomial sequence $\left(  g_i (z) \right)_{i=1}^{p}$ by 
   \begin{align}\label{eq-f-g}
   	 g_i(z) = z \sum_{j=1}^{t_i-1}  f_j(z) + a_i f_{t_i}(z) + \sum_{j = t_i+1}^{q} f_j(z),
   \end{align}
   where $a_k$ is $1$ or $1+z$ and $1 \le t_1 \le \cdots \le t_{p} \le q$.
   Also, if $t_i=t_j$ for some $i<j$, then $a_i=1+z$ and $a_j=1$ can not happen at the same time.
	If the sequence $\left( f_i (z)\right)_{i=1}^{q}$ is interlacing, then so is $\left(g_i (z)\right)_{i=0}^{p}$.
\end{thm}

\begin{proof}
Define  a  $p \times q$ matrix $G= (G_{ij}(z))$ as
\begin{align*}
G_{ij}(z)=	\begin{cases}
	z, & \mbox{if }j < t_i, \\
	a_k, & \mbox{if } j = t_i, \\
	1, & \mbox{if } j > t_k. \\
	\end{cases}
\end{align*}
Then clearly $G_{ij}(z)$ has nonnegative coefficients for any $1\le i \le p$ and $1\le j \le q$.
As shown by Br{\"a}nd{\'e}n~\cite[Corollary 8.7]{Braenden2015Unimodality}, every $2\times2$ submatrices  of $G$ with entries $1$ and $z$ only satisfies \eqref{pbxvw}.
Hence, it suffices to consider the cases for all the possible $2\times2$ submatrices of $G$ where the entry $1+z$ appears. Instead of directly checking \eqref{pbxvw}, we shall  prove these $2\times2$ submatrices preserve interlacing, which by Lemma~\ref{pb} is equivalent to \eqref{pbxvw} for the  $2\times2$  submatrices of polynomials with nonnegative coefficients. 

We first prove the  matrix $\left(\begin{array}{cccc}
1+z & 1 \\[3pt]
z & 1+z
\end{array}\right)$ preserves interlacing. Assume that $f$ and $g$ are two real-rooted polynomials with nonnegative coefficients satisfying $f\ll g$. Then, $(1+z)f \ll zf$ and $g \ll zf$ and hence $(1+z)f+g \ll zf$ by Lemma~\ref{lem:convex}. Similarly, $(1+z)f+g \ll (1+z)g$. Therefore, it follows from Lemma~\ref{lem:convex}  that $(1+z)f+g \ll zf+(1+z)g$.
	
We next consider the remaining cases in a unified approach.	
These $2\times 2$ matrices can be written as 
	\allowdisplaybreaks
	\begin{align*}
		\left(\begin{array}{cccc}
		1+z & 1 \\[3pt]
		z & 1 
		\end{array}\right)	& = 
		\left(\begin{array}{cccc}
		1 & 1  \\[3pt]
		0 & 1 
		\end{array}\right)
		\left(\begin{array}{cccc}
		1 & 0 \\[3pt]
		z & 1 
		\end{array}\right), \\[8pt]
		\left(\begin{array}{cccc}
		1 & 1 \\[3pt]
		1+z & 1 
		\end{array}\right)	& = 
		\left(\begin{array}{cccc}
		1 & 0  \\[3pt]
		1 & 1 
		\end{array}\right)
		\left(\begin{array}{cccc}
		1 & 1 \\[3pt]
		z & 0
		\end{array}\right), \\[8pt]
		\left(\begin{array}{cccc}
		1+z & 1 \\[3pt]
		z & z 
		\end{array}\right)	& = 
		\left(\begin{array}{cccc}
		1 & 1  \\[3pt]
		z & 0 
		\end{array}\right)
		\left(\begin{array}{cccc}
		1 & 1 \\[3pt]
		z & 0
		\end{array}\right), \\[8pt]
		\left(\begin{array}{cccc}
		1 & 1 \\[3pt]
		z & 1+z 
		\end{array}\right)	& = 
		\left(\begin{array}{cccc}
		1 & 0  \\[3pt]
		z & 1 
		\end{array}\right)
		\left(\begin{array}{cccc}
		1  & 1 \\[3pt]
		0 & 1
		\end{array}\right), \\[8pt]
		\left(\begin{array}{cccc}
		z & 1+z \\[3pt]
		z & z 
		\end{array}\right)	& = 
		\left(\begin{array}{cccc}
		1 & 1  \\[3pt]
		0 & 1 
		\end{array}\right)
		\left(\begin{array}{cccc}
		0 & 1 \\[3pt]
		z & z
		\end{array}\right), \\[8pt]
		\left(\begin{array}{cccc}
		z & 1 \\[3pt]
		z & 1+z 
		\end{array}\right)	& = 
		\left(\begin{array}{cccc}
		1 & 0  \\[3pt]
		1 & 1 
		\end{array}\right)
		\left(\begin{array}{cccc}
		z & 1 \\[3pt]
		0 & z
		\end{array}\right), \\[8pt]
		\left(\begin{array}{cccc}
		z & 1+z \\[3pt]
		z & 1+z 
		\end{array}\right)	& = 
		\left(\begin{array}{cccc}
		1 & 1 \\[3pt]
		1 & 1 
		\end{array}\right)
		\left(\begin{array}{cccc}
		z & 1 \\[3pt]
		0 & z
		\end{array}\right), \\[8pt]
		\left(\begin{array}{cccc}
		1+z & 1 \\[3pt]
		1+z & 1 
		\end{array}\right)	& = 
		\left(\begin{array}{cccc}
		1 & 1 \\[3pt]
		1 & 1 
		\end{array}\right)
		\left(\begin{array}{cccc}
		1 & 0 \\[3pt]
		z & 1
		\end{array}\right).
	\end{align*}
All the matrices  on the right hand side preserve interlacing, which has already been checked in~\cite{Leander2016Compatible, Zhang2016Real}. So do  the matrices  on the left hand side.
This completes the proof by Lemma~\ref{pb}.
\end{proof}
We note that  an interlacing-preserving matrix  
 has no $2\times2$ submatrices of  the form  $$\left(\begin{array}{cc}
	1+z & 1 \\[3pt]
	1+z & 1+z 
\end{array}\right)  \quad   \mbox{ or } \quad  \left(\begin{array}{cc}
1+z & 1+z \\[3pt]
z & 1+z 
\end{array}\right).$$
Indeed, two counterexamples are given below:
\begin{eqnarray*}
& \left(\begin{array}{cc}
1+z & 1 \\[3pt]
1+z & 1+z 
\end{array}\right)  \
\left(\begin{array}{c}
1+z  \\[3pt]
z 
\end{array}\right)    & = \ \
\left(\begin{array}{c}
1+3z+z^2 \\[3pt]
1+3z+2z^2
\end{array}\right),\\[8pt]
& \left(\begin{array}{cc}
1+z & 1+z \\[3pt]
z & 1+z 
\end{array}\right)  \ \
\left(\begin{array}{c}
1  \\[3pt]
z 
\end{array}\right)  & =  \ \
\left(\begin{array}{c}
1+2z+z^2 \\[3pt]
2z+z^2
\end{array}\right).
\end{eqnarray*}


Theorem~\ref{thm: interlacing recursion} allows us to prove Theorem~\ref{thm:s}.

\begin{proof}[Proof of Theorem~\ref{thm:s}]
	We shall prove the first part of this theorem by induction on $m$.
For $m=1$, 	 the initial conditions imply that the polynomial sequence 
$$\left( p_{1,k}^\mathbf{s}(z) \right)_{k=0}^{s_1-1}= \left( 1+z, z,\ldots, z \right)$$ 
is interlacing.
Since the recurrence relation~\eqref{eqn: recursion} satisfies the condition of Theorem~\ref{thm: interlacing recursion}, by  induction on $m$, we obtain that  the polynomial sequence $\left( p_{m,k}^\mathbf{s}(z) \right)_{k=0}^{s_m-1}$ is interlacing for all $1 \le m\le n$. 
This proves the first part of Theorem~\ref{thm:s}.

We proceed to prove the second part.
By the above paragraph, we know  the polynomial sequence $\left( p_{n,k}^\mathbf{s}(z) \right)_{k=0}^{s_n-1}$ is interlacing. Hence, $p_{n,k}^\mathbf{s}(z)  \ll z \, p_{n,0}^\mathbf{s}(z)$ for all $0 \le k \le s_{n-1}$.
Therefore, the polynomial 
	\begin{align*}
	\widetilde{E}_n^\mathbf{s}(z)  & = \ (1+z)p_{n,0}^\mathbf{s}(z) +   \sum_{k=1}^{s_n-1}p_{n,k}^\mathbf{s}(z)  \\
			& =\sum_{k=0}^{s_n-1}p_{n,k}^\mathbf{s}(z) + z \, p_{n,0}^\mathbf{s}(z)
	\end{align*}
	is a real-rooted polynomial.  
	This completes the proof of Theorem~\ref{thm:s}.
\end{proof}

\section{Applications}\label{sect-applications}
In this section, we shall show that Theorem~\ref{thm:s} contains  several  real-rootedness results as special cases,  which appear to be new. 
All these results parallel applications of $\mathbf{s}$-Eulerian polynomials  and $\mathbf{s}$-derangement polynomials in \cite{Savage2015$s$, Gustafsson2018Derangements}.

\subsection{Binomial Eulerian polynomials for permutations}
In this subsection, we shall prove the real-rootedness of $\widetilde{A}_n(z)$.

\begin{thm} \label{thm-perm} 
	For any positive integer $n,$ 
	the binomial Eulerian polynomial $\widetilde{A}_{n}(z)$ has only real roots.
\end{thm}

\begin{proof}
Athanasiadis~\cite[(32)]{Athanasiadis2018Binomial} showed  that 
\begin{equation*}
\widetilde{A}_n(z) = \sum_{k=0}^{n} \binom{n}{k} (1+z)^{n-k} d_k(z).
\end{equation*}
Hence, it follows that
\begin{align*}
\widetilde{A}_{n}(z)  = \sum_{\pi \in \fS_n} (1+z)^{\fix (\pi)} z^{\exc (\pi)},
\end{align*}
where $\fix (\pi)$ is the number of fixed points in $\pi$, namely,
$\fix (\pi) = |\{i: \pi_i=i \}|$.
Following the bijection given by Steingrimsson~\cite[Appendix]{Steingrimsson1994Permutation}, one can see that 
\begin{align*}
\widetilde{A}_{n}(z)  = \sum_{\pi \in \fS_n} (1+z)^{\bad (\pi)} z^{\des (\pi)},
\end{align*}
where
$\bad (\pi)$ is the number of $\pi_k$'s such that $\pi_k<\pi_m$ for all $m>k$ and $\pi_{k-1}<\pi_k$ with the assumption $\pi_0=0$.

For a permutation $\pi\in\mathfrak{S}_n$ and $i\in[n]$, we let 
$t_i :=|\{j > i \, : \, \pi_j<\pi_i\}|$
denote the number of inversions of $\pi$ at $i$.  
Clearly, $0 \le t_i\le n-i$.
The sequence $t = (t_1, t_2, \ldots,t_n)$ is called the Lehmer code of $\pi$.
Define a map  $$\Theta : \mathfrak{S}_n \longrightarrow \mathcal{I}_{n-1}^{(2,3, \ldots,n)}$$
by letting $$\mathbf{e} = \Theta (\pi_1 \pi_2 \cdots\pi_n) =  (t_{n-1}, \ldots, t_2, t_1).$$
It is known that $\Theta$ is a bijection and  $\des (\pi) = \asc (\mathbf{e})$. We note that $\bad (\pi) = \col (\mathbf{e})$.
This is because if  $\frac{t_{j-1}}{n-j+1} = \frac{t_j}{n-j}$ for $2\le j \le n$ then it must be that $t_{j-1}= t_j=0$, which implies that $\pi_j$ is bad,  and vice versa. Besides, $t_1=0$ is equivalent to saying that  $\pi_1$ equals $1$  and hence is bad. 
Therefore, $\widetilde{A}_{n}(z)$ is a special case of the polynomial $\widetilde{E}_{n-1}^\mathbf{s}(z)$ when $\mathbf{s} = (2,3,  \ldots, n)$.  
This completes the proof.
\end{proof}

We remark that the real-rootedness of $A_n(z)$, $d_n(z)$, and $\widetilde{A}_n(z)$ can be proved in a unified approach. For $\alpha=0,1,1+z$, we define a sequence of polynomials $(A^{\alpha}_{n,0}(z))_{i=0}^{n-1}$ as follows:
\begin{align*}
A^{\alpha}_{n,i}(z)   = \sum_{\pi \in \fS_n \atop \pi_{1}=i+1} \alpha^{\bad'(\pi)} z^{\des (\pi)}
\end{align*}
where $\bad'(\pi)$ is  $\bad(\pi)-1$ if $\pi_1=1$ and $\bad(\pi)$  otherwise.
One can see that these polynomials satisfy the following recurrence relation:
\begin{align*}\label{A-recurrence-matrix}
\left(\begin{array}{c}
A^{\alpha}_{n,0}(z) \\[3pt]
A^{\alpha}_{n,1}(z) \\[3pt]
\vdots \\[3pt]
A^{\alpha}_{n,n-2}(z) \\[3pt]
A^{\alpha}_{n,n-1}(z)
\end{array}\right)
=
\left(\begin{array}{cccc}
\alpha & 1 & \cdots & 1 \\[3pt]
z & 1 & \cdots & 1 \\[3pt]
\vdots & \vdots &   & \vdots \\[3pt]
z & z & \cdots & 1 \\[3pt]
z & z & \cdots & z
\end{array}\right)
\left(\begin{array}{c}
A^{\alpha}_{n-1,0}(z) \\[4pt]
A^{\alpha}_{n-1,1}(z) \\[4pt]
\vdots \\[4pt]
A^{\alpha}_{n-1,n-2}(z)
\end{array}\right)
\end{align*}
with the initial condition $A^{\alpha}_{1,0}(z)=1$.
Hence,  we know that  the polynomial sequence $(A^{\alpha}_{n,0}(z))_{i=0}^{n-1}$ is interlacing, and therefore $A^{\alpha}_{n+1,0}(z)$ corresponding to $A_n(z)$, $d_n(z)$, and $\widetilde{A}_n(z)$ for $\alpha=1,0, 1+z$, respectively,  has only real roots.

\subsection{Binomial Eulerian polynomials for colored permutations}

For nonnegative integers $m$ and $n$, let $[m,n] = \{m,  m+1, \cdots , n\}$.
For positive integers $n$ and $r$, an \emph{$r$-colored permutation},  introduced by Steingr\'imsson \cite{Steingrimsson1992Permutations, Steingrimsson1994Permutation},  is a pair $\pi^c$, where $\pi\in\mathfrak{S}_n$ and $c\in [0,r-1]^n$,  usually denoted as $\pi_1^{c_1}\pi_2^{c_2}\cdots\pi_n^{c_n}$. 
Denote  by $\Z_{r}\wr\mathfrak{S}_n$ the set of $r$-colored permutations.
An index $i\in[n]$ is said to be  a descent in $ \pi^c$ if either $c_i>c_{i+1}$ or $c_i = c_{i+1}$ and $\pi_i>\pi_{i+1}$, with the assumption that $\pi_{n+1} = n+1$ and $c_{n+1} = 0$.  
An index $i\in[n]$ is said to be an excedance of $ \pi^c$ if either $\pi_i>i$ or $\pi_i = i$ and $c_i>0$.  
Denote  by $\des (\pi^c)$ and $\exc (\pi^c)$ the number of descents and excedances in $\pi^c$, respectively.  
A colored permutation $\pi^c$ is called a \emph{derangement} if it has no fixed points of color $0$, and denote by $\mathfrak{D}_{n,r}$ the subset consisting of derangements in $\Z_{r}\wr\mathfrak{S}_n$. 
The colored permutation analogues of the Eulerian polynomials and derangement polynomials are  defined   as follows:
\[
A_{n,r}(z) :=\sum_{\pi^c\in\Z_{r}\wr\mathfrak{S}_n}z^{\des (\pi^c)}
\qquad
\mbox{and}
\qquad
d_{n,r}(z) :=\sum_{\pi^c\in\mathfrak{D}_{n,r}}z^{\exc (\pi^c)},
\]
respectively.
The real-rootedness of the colored Eulerian polynomials $A_{n,r}(z)$ was proved by Steingr\'imsson~\cite{Steingrimsson1992Permutations, Steingrimsson1994Permutation}.
The real-rootedness of  derangement polynomials of type $B$  $d_{n,2}(z)$ was proved by Chen, Tang, and Zhao~\cite{Chen2009Derangement}, and by Chow~\cite{Chow2009derangement}, independently.
 Athanasiadis~\cite{Athanasiadis2014Edgewise} showed that $d_{n,r}(z)$ can be expressed as $$d_{n,r}(z) = d^{+}_{n,r}(z)+ d^{-}_{n,r}(z),$$
where $d^{+}_{n,r}(z)$ and $d^{-}_{n,r}(z)$ are $\gamma$-positive polynomials with centers of symmetry $\frac{n}{2}$ and $\frac{n+1}{2}$, respectively.
Such a decomposition is called  the  \emph{symmetric decomposition} of polynomials by Br\"andÃ\'en and Solus~\cite{Braenden2018Symmetric}.
Recently, Gustafsson and Solus~\cite{Gustafsson2018Derangements} proved that both $d^{+}_{n,r}(z)$ and $d^{-}_{n,r}(z)$  have only real roots, and   Br\"and\'en and Solus~\cite{Braenden2018Symmetric} further proved that $d^{+}_{n,r}(z) \ll d^{-}_{n,r}(z)$.

Recently, Athanasiadis~\cite{Athanasiadis2018Binomial}  introduced a generalization 
$\widetilde{A}_{n,r}(z)$ of $\widetilde{A}_n(z)$ to 
the wreath product group $\ZZ_r \wr \fS_n$ and 
further studied their symmetric function generalizations.  
The polynomial $\widetilde{A}_{n,r}(z)$ is defined by the formula
\begin{equation*}
\widetilde{A}_{n,r}(z) \ := \ \sum_{m=0}^n {n \choose m} z^{n-m} A_{m,r}(z).
\end{equation*}
Athanasiadis~\cite{Athanasiadis2018Binomial} also studied the  symmetric decomposition of  $\widetilde{A}_{n,r}(z)$ as 
$$\widetilde{A}_{n,r}(z)  =\widetilde{A}^+_{n,r}(z) + \widetilde{A}^-_{n,r}(z), $$
where $\widetilde{A}^+_{n,r}(z)$ and $\widetilde{A}^-_{n,r}(z)$ are two  $\gamma$-positive polynomials which can be defined by
\begin{align}
\widetilde{A}^+_{n,r}(z)  \ : = \   \sum_{k=0}^n 
{n \choose k} (1+z)^{n-k} d^+_{k,r}(z),  \label{eq:wA+}\\
\widetilde{A}^-_{n,r}(z)  \ : = \   \sum_{k=0}^n 
{n \choose k} (1+z)^{n-k} d^-_{k,r}(z). \label{eq:wA-}
\end{align}
In this subsection, we shall prove the real-rootedness of  this symmetric decomposition.
\begin{thm} \label{thm-colored} 
	For positive integers $n, r$ with $r \ge 2$ 
	we have that $\widetilde{A}^+_{n,r}(z) \ll \widetilde{A}^-_{n,r}(z)$ and hence 
	$\widetilde{A}_{n,r}(z) =  	\widetilde{A}^+_{n,r}(z) +  	\widetilde{A}^-_{n,r}(z)$ has only real roots.
\end{thm}

In order to prove Theorem~\ref{thm-colored}, we first give a combinatorial explanation for $\widetilde{A}^+_{n,r}(z)$ and $\widetilde{A}^+_{n,r}(z)$.
Denote  by $(\ZZ_r \wr \fS_n)^+$ and $(\ZZ_r \wr \fS_n)^-$  the set of colored 
permutations $\pi^c \in \ZZ_r \wr \fS_n$ with the last 
coordinate of zero color and nonzero color, respectively.  
Following~\cite{Gustafsson2018Derangements}, given a colored permutation $\sigma = \pi^c \in \Z_{r}\wr\mathfrak{S}_n$, an element $i \in [n]$ is said to be bad with respect to $\sigma$ if for $\pi_j = i$ it holds that
\begin{enumerate}
	\item $\pi_j < \pi_k$ for every $k > j$,
	\item $\pi_{j-1} < \pi_k$ for every $k > j-1$, and
	\item $\pi_j$ and $\pi_{j-1}$ have the same color,
\end{enumerate}
with  the convention $\pi_0 = 0$ and $c_0 = 0$.  
Let $S_{\sigma}$ be the set of bad elements in $\sigma \in \ZZ_r \wr \fS_n$ and denote $\bad (\sigma):= |S_{\sigma}|.$
A combinatorial interpretation of $\widetilde{A}^+_{n,r}(z)$ and $\widetilde{A}^-_{n,r}(z) $ is stated as follows.
\begin{lem}
	For positive integers $n$ and $r$, we have that 
	\begin{align*}
		\widetilde{A}^+_{n,r}(z) & \  : =  \sum_{w \in (\ZZ_r \wr \fS_n)^+} (1+z)^{\bad (w)}	z^{\des(w)}, \\[6pt]
		\widetilde{A}^-_{n,r}(z)  & \  : =  \sum_{w \in (\ZZ_r \wr \fS_n)^-}   (1+z)^{\bad (w)} z^{\des(w)}.
	\end{align*}
\end{lem}
\begin{proof}	
Our proof is closely related to that of \cite[Theorem~4.6]{Gustafsson2018Derangements}.
It is known that the symmetric decomposition of a given polynomial is uniquely determined.
From the proof of \cite[Theorem~4.6]{Gustafsson2018Derangements} and  \cite[Lemma~4.3.9]{Gustafsson2018Box},  we obtain that  
\begin{align*}
d^+_{n,r}(z) \ = \sum_{\sigma \in (\ZZ_r \wr \fS_n)^+ \atop S_{\sigma}=\emptyset} 	z^{\des (\sigma)}
\qquad \mbox{ and } \qquad
d^-_{n,r}(z) \ = \sum_{\sigma \in (\ZZ_r \wr \fS_n)^- \atop S_{\sigma}=\emptyset} 	z^{\des (\sigma)}.
\end{align*}

In order to prove this lemma, we proceed to describe a way to construct all the  colored permutations from permutations with $S_\sigma=\emptyset$.
The construction is as follows, which is similar to that in the proof of~\cite[Theorem~4.6]{Gustafsson2018Derangements}.

\begin{enumerate}
	\item Take an element $\sigma = \pi_1^{c_1}\pi_2^{c_2}...\pi_{k}^{c_{k}} \in \Z_r\wr\mathfrak{S}_{k}$ with $S_\sigma=\emptyset$ and choose a subset $T$ of $[n]$ with cardinality $n-k$.
	\item Replace each element $\pi_i = j$ with the $j${th} smallest element of $[n]\setminus T$. 
	\item We will now insert the elements in $T$ into the permutation obtained in the previous step, in such a way which make them bad. Pick each element $i\in T$ in the relative order from the smallest to the largest. If $i = 1$, insert $i$ at the front of $\sigma$ and give it color $0$. Otherwise, find the rightmost element $\pi_j$ such that $\pi_j < i$ and $\pi_j < \pi_k$ for every $k > j$. Give $i$ the same color as $\pi_j$ and insert it right after $\pi_j$.
\end{enumerate}

Such a construction does not affect the descent number and  makes the set $T$ correspond to the set of bad elements $S_\sigma$.  Hence, the desired combinatorial identities of $\widetilde{A}^+_{n,r}(z)$ and $\widetilde{A}^-_{n,r}(z)$ follow immediately  from their formal definitions in~\eqref{eq:wA+} and  \eqref{eq:wA-}, which completes the proof.
\end{proof}

As we now describe, the statistics on $\ZZ_r \wr \fS_n$ are related to statistics on $\mathbf{s}$-inversion sequences $\mathcal{I}_n^\mathbf{s}$ with $\mathbf{s}= (rn,\ldots, 2r,r)$.

\begin{lem}
For positive integers $n$ and $r$, we have that 
	\begin{align*}
		\widetilde{A}^+_{n,r}(z) \ = \ p_{n,0}^\mathbf{s}(z)  
		\quad \mbox{ and } \quad
		\widetilde{A}^-_{n,r}(z) \  =  \ \sum_{k=1}^{r-1} p_{n,k}^{\mathbf{s}}(z),
	\end{align*}
where	$\mathbf{s}= (rn,\ldots, 2r,r)$.
	In particular, $\widetilde{A}^+_{n,r}(z) = \widetilde{E}_{n-1}^{(rn,\ldots, 2r)}(z)$.
\end{lem}
\begin{proof}
Let $\mathbf{s}= (rn,\ldots, 2r,r)$.
We define $\Psi: \Z_{r}\wr\mathfrak{S}_n\longrightarrow \mathcal{I}_n^{\mathbf{s}}$ where
\[
\Psi: \pi_1^{c_1}\cdots\pi_n^{c_n} \longmapsto (nc_1+t_1, \ldots,2c_{n-1}+t_{n-1},c_n+t_n).
\]
The inverse mapping $\Psi^{-1}$ is given by 
\[
\Psi^{-1}: (e_1,\ldots,e_n)\longmapsto \pi_1^{c_1}\cdots\pi_n^{c_n},
\]
where $\pi_1\cdots\pi_n$ is the permutation with inversion sequence
\[
t = (e_{1}-nc_1,\ldots,e_{n-1}-2c_{n-1}, e_n-c_n),
\]
and 
$
c_i = \left\lfloor\frac{e_{i}}{n-i+1}\right\rfloor
$ 
for each $i\in[n]$.
Note that $t_n=0$ and hence we have that
\[
\Psi^{-1}(\{\mathbf{e} \in \mathcal{I}_n^{\mathbf{s}}   :   e_n >0  \}) = \{\sigma \in \Z_{r}\wr\mathfrak{S}_n \, : \,   c_n > 0 \},
\]
and
\[
\Psi^{-1}(\{\mathbf{e} \in \mathcal{I}_n^{\mathbf{s}}   :   e_n = 0  \})    = \{\sigma \in \Z_{r}\wr\mathfrak{S}_n \, : \,   c_n = 0 \}.
\]
It is clear to see that $\des (\pi^c) = \des (\mathbf{e})$. 
We shall prove that $\bad (\pi^c) = \col' (\mathbf{e})$. 
For convenience, we let $t_0=0$.
For $1\le j \le n$, if $\frac{e_{j-1}}{(n-j+1)r} = \frac{e_{j}}{(n-j)r}$ then it must be that $c_{j-1}= c_j$ and $t_{j-1}= t_j=0$, which implies that $\pi_j$ is bad. The converse statement is true as well.
Hence we get that 
\begin{align*}
\widetilde{A}^+_{n,r}(z) & =   	\sum_{\mathbf{e} \in \mathcal{I}_n^\mathbf{s}}\chi(e_n = 0)  (1+z)^{\col' (\mathbf{e})}  z^{\des(\mathbf{e})}, \\[6pt]
\widetilde{A}^-_{n,r}(z)  & =   \sum_{\mathbf{e} \in \mathcal{I}_n^{\mathbf{s}}}\chi(e_n > 0)  (1+z)^{\col' (\mathbf{e})}  z^{\des(\mathbf{e})}.
\end{align*}
Let $f : \mathcal{I}_n^\mathbf{s} \rightarrow \mathcal{I}_n^\mathbf{s}$ be the involution defined by 
$
f(\mathbf{e})_i = -e_i \text{ mod } s_i.
$
It was shown in~\cite[Theorem~3.1]{Gustafsson2018Derangements} that $\asc (\mathbf{e}) = \des (f( \mathbf{e}))$ and $\des(\mathbf{e}) = \asc (f(\mathbf{e}))$. Since $\asc (\mathbf{e}) + \col  (\mathbf{e}) + \des(\mathbf{e}) =n$ for any $\mathbf{e}  \in \mathcal{I}_n^\mathbf{s}$,
it follows that $\col  (\mathbf{e})  = \col  (f(\mathbf{e}))$ and thus $\col' (\mathbf{e})  = \col' (f(\mathbf{e}))$ 
Therefore, we get that
	\begin{align*}
	\widetilde{A}^+_{n,r}(z) & =   	\sum_{\mathbf{e} \in \mathcal{I}_n^\mathbf{s}}\chi(e_n = 0)  (1+z)^{\col' (\mathbf{e})}  z^{\asc(\mathbf{e})}, \\[6pt]
	\widetilde{A}^-_{n,r}(z)  & =   \sum_{\mathbf{e} \in \mathcal{I}_n^{\mathbf{s}}}\chi(e_n > 0)  (1+z)^{\col' (\mathbf{e})}  z^{\asc(\mathbf{e})}.
\end{align*}
This completes the proof.
\end{proof}
Now we are in the position to prove Theorem~\ref{thm-colored}.

\begin{proof}[Proof of Theorem~\ref{thm-colored}]
Let $\mathbf{s}= (rn,\ldots, 2r,r)$.
By Theorem~\ref{thm:s}, the sequence $\left( p_{n,k}^\mathbf{s}(z) \right)_{k=0}^{r-1}$ is interlacing. 
By Lemma~\ref{lem:convex}, we get that 
$$\widetilde{A}^+_{n,r}(z)\ =\ p_{n,0}^\mathbf{s}(z) \ \ll \ \ \sum_{k=1}^{r-1} p_{n,k}^{\mathbf{s}}(z) = \widetilde{A}^-_{n,r}(z).$$
Since $\widetilde{A}_{n,r}(z) =  	\widetilde{A}^+_{n,r}(z) +  	\widetilde{A}^-_{n,r}(z)$, the real-rootedness of $\widetilde{A}_{n,r}(z)$ follows immediately. This completes the proof.
\end{proof}

\subsection{The  Edgewise Subdivision}
\label{subsec: edgewise subdivision}

Athanasiadis~\cite{Athanasiadis2012Flag} considered a triangulation of a sphere from  a triangulation of a simplex.
Let $V = \{v_1, v_2,\dots,v_n\}$ be an $n$-element 
set and let $\Gamma$ be a triangulation of the simplex 
$2^V$. Define $U = \{u_1, 
u_2,\dots,u_n\}$ to a new $n$-element set.  
Denote by  $\Delta(\Gamma)$ the 
sets of  the form $E \cup G$, where 
$E = \{ u_i: i \in I\}$ is a face of the simplex 
$2^U$ for some $I \subseteq [n]$ and $G$ is a face 
of the restriction $\Gamma_F$ of $\Gamma$ to the 
face $F = \{ v_i: i \in [n] \backslash I\}$ of the simplex 
$2^V$. 
Athanasiadis~\cite{Athanasiadis2012Flag, Athanasiadis2018Binomial} showed that the  $h$-polynomial of $\Delta(\Gamma)$ satisfies
\begin{align*}
h(\Delta(\Gamma), z) \ = \ \sum_{F \subseteq V}  t^{n-|F|} \, h(\Gamma_F, z),
\end{align*}
and
\begin{align*}
h(\Delta(\Gamma), z) \ = \ \sum_{F \subseteq V}  (1+z)^{n-|F|}  \ell_F (\Gamma_F, z),
\end{align*} 
where $ \ell_F(\, \cdot\, , z)$ is the local $h$-polynomial introduced by Stanley~\cite{Stanley1992Subdivisions}.
It is natural to ask  how $h$-polynomials $h(\Delta(\Gamma),z)$ relate as $\widetilde{E}_n^\mathbf{s}(z)$ for certain integer sequences. 
In this subsection, we provide a new class of such  $h$-polynomials when the $\Gamma$ are edgewise subdivisions of simplexes.

The edgewise subdivision is a well-studied subdivision of a simplicial complex that arises in a variety of mathematical contexts, see \cite{Beck2010log, Brenti2009Veronese, Brun2005Subdivisions, Eisenbud1994Initial, Jochemko2018Combinatorial, Kubitzke2012Enumerative}.
One of its properties is that its faces $F$  are divided into $r^{\dim(F)}$  faces of the same dimension.
Athanasiadis \cite{Athanasiadis2014Edgewise, Athanasiadis2016local} showed that
\begin{align}\label{eq:localedge}
\ell_V \left((2^V)^{\langle r \rangle}, x \right)=
\ {\rm E}_r \left(\, (x + x^2 + \cdots + x^{r-1})^n \right),
\end{align}
where ${\rm E}_r$ is a linear operator defined on polynomials by setting ${\rm E}_r (x^n) = x^{n/r}$, if $r$ divides $n$,
and ${\rm E}_r (x^n) = 0$ otherwise.
\begin{lem}
		For positive integers $r$ and $n$,  
\begin{align}\label{eq-h-esd}
	h \left(\Delta \left( \esd_r \left( 2^{[n]}\right) \right), z  \right)  = {\rm E}_r \left(  (1 + z + z^2 + \cdots + z^{r})^n \right).
\end{align}
\end{lem}
\begin{proof}	
It is known that \cite{Athanasiadis2016local}, for any simplicial complex $\Delta$, the $r$-fold edgewise subdivision $\esd_r(\Delta)$ restricts to $\esd_r(2^F)$ for every $F\in \Delta$.
Hence, it follows that 
\begin{align}\label{eq:h-esd}
h \left(\Delta \left( \esd_r \left(2^{[n]}\right) \right), z  \right)  = \sum_{k=0}^{n} \binom{n}{k} (1+z)^{n-k} \, \ell \left(\esd_r(2^{[k]}) , z\right).
\end{align}
Athanasiadis~\cite{Athanasiadis2016local} also proved that  the local $h$-polynomial of the $r$th edgewise subdivision of a simplex is 
\[
\ell \left(\esd_r(2^{[k]}) , z\right) 
\, =\, {\rm E}_r \left(  ( z + z^2 + \cdots + z^{r-1})^k \right).
\]
Therefore, we have that
\begin{align*}
h \left(\Delta \left( \esd_r \left( 2^{[n]}\right) \right), z  \right)  &  = \sum_{k=0}^{n} \binom{n}{k} (1+z)^{n-k} \, {\rm E}_r \left( ( z + z^2 + \cdots + z^{r-1})^k \right)  \\[6pt]
&  ={\rm E}_r \left(  \sum_{k=0}^{n} \binom{n}{k}  \,  (1+z^r)^{n-k} (z + z^2 + \cdots + z^{r-1})^k \right)  \\[6pt]
&  ={\rm E}_r \left(  (1 + z + z^2 + \cdots + z^{r})^n \right).
\end{align*}
This completes the proof.
\end{proof}


Let $\mathcal{W}(n,r)$ denote the set of words $w = w_0 w_1\cdots w_n$ where $w_i\in[0,r-1]$  for all $1\le i \le n-1$, with the assumption $w_0 = w_n = 0$.  
Given a  word $w \in \mathcal{W}(n,r)$, an index $i\in[0,n]$ is said to be
an  ascent if $w_i<w_{i+1}$, and 
a  collision if $w_i=w_{i+1}$, and 
we let $\asc (w)$ and $\col (w)$ denote the number of ascents and collisions in $w$, respectively. 
Let $\SW(n,r)$ denote the subset of words in $\mathcal{W}(n,r)$ with no collisions.
The words in  $\SW(n,r)$ are called Smirnov words, see \cite{Athanasiadis2017survey}.
\begin{thm}For a positive integer $r$ and $n$,  we have 
	\begin{align}\label{eq-h-w}
	h \left(\Delta \left( \esd_r \left( 2^{[n]}\right) \right), z  \right) = 
	\sum_{w \in \mathcal{W}(n+1,r)} (1+z)^{\col(w)} z^{\asc (w)}.
	\end{align}
	Moreover, this polynomial has only real roots.
\end{thm}
\begin{proof}
	It is not hard to see that 
	\begin{align*}
	\sum_{w \in \mathcal{W}(n+1,r)} (1+z)^{\col(w)} z^{\asc (w)}   \, =\, \sum_{k=0}^{n} \binom{n}{k} (1+z)^{n-k}  \sum_{w\in \mathcal{SW}(k+1,r)}  z^{\asc (w)}.
	\end{align*}
Since it is known~\cite{Athanasiadis2016local} that 
	\begin{align*}
	\ell \left(\esd_r(2^{[k]}) , z\right) = 	 \sum_{w\in \mathcal{SW}(n+1,r)}  z^{\asc (w)},
	\end{align*}
the identity~\eqref{eq-h-w} follows from~\eqref{eq:h-esd}.

	Let $\mathbf{s} = (r,r,\ldots,r)$.
From the definition of $\mathcal{W}(n+1,r)$ and that of $\mathcal{I}_n^{\mathbf{s}}$, we know that 
\[
\widetilde{E}_n^{\mathbf{s}}(z) = \sum_{w\in \mathcal{W}(n+1,r)} (1+z)^{\col(w)} z^{\asc (w)}.
\]
Thus, their real-rootedness  follows  immediately from Theorem~\ref{thm-perm}. 
\end{proof}

Before ending this subsection, we shall present an alternative way to prove the real-rootedness of $h \left(\Delta \left( \esd_r \left( 2^{[n]}\right) \right), z  \right)$.
Given a polynomial $f(z)$, there exist uniquely determined polynomials  $f^{\langle r,0 \rangle}(z)$, $f^{\langle r,1 \rangle}(z)$, $ \ldots,$  $f^{\langle r,r-1 \rangle}(z)$ such that
\begin{align*}
f(z) = f^{\langle r,0 \rangle}(z^r) + z f^{\langle r,1 \rangle}(z^r)+ \cdots +z^{r-1}f^{\langle r,r-1 \rangle}(z^r).
\end{align*}
The alternative proof is based on \eqref{eq-h-esd} and the following lemma.
\begin{lem}\label{lem-f}
	Let $r$ be a  positive integer.
	Suppose that $f(z)$ and
	$g(z)$ are two polynomial  with nonnegative coefficients  satisfying
	\begin{align}
	\left(1+z+\dots+z^{r}\right) f(z)  = g(z). \label{eq:fg}
	\end{align}
	If the sequence $\left(f^{\langle r, r-1 \rangle}(z), \ldots, f^{\langle  r,1 \rangle}(z), f^{\langle r,0 \rangle}(z)\right)$ is  interlacing, then so is $\big( g^{\langle r, r-1 \rangle}(z)$, $\ldots, g^{\langle  r,1 \rangle}(z), g^{\langle r,0 \rangle}(z) \big)$.
\end{lem}
\begin{proof}
The identity~\eqref{eq:fg} can be expressed in a matrix form as follows:
\begin{align}\label{eq-m}
\left(\begin{array}{c}
g^{\langle r, r-1 \rangle}(z) \\[3pt]
g^{\langle r, r-2 \rangle}(z) \\[3pt]
\vdots \\[3pt]
g^{\langle r, 0 \rangle}(z)
\end{array}\right)
=
\left(\begin{array}{cccc}
1+z & 1 & \cdots  &1 \\[3pt]
z & 1+z & \cdots  &1 \\[3pt]
\vdots & \vdots &   & \vdots \\[3pt]
z & z & \cdots & 1+z
\end{array}\right)
\left(\begin{array}{c}
f^{\langle r, r-1 \rangle}(z) \\[3pt]
f^{\langle r, r-2 \rangle}(z) \\[3pt]
\vdots \\[3pt]
f^{\langle r, 0 \rangle}(z)
\end{array}\right).
\end{align}
By Theorem~\ref{thm: interlacing recursion}, we get the desired result.
This completes the proof.
\end{proof}

Note that the transforming matrix in~\eqref{eq-m} also appears in~\cite[Lemma~4.4]{Solus2019Simplices}.
By iteratively using the above theorem, we obtain the following result. 

\begin{cor}\label{main2}
	Let $r$ be a positive integer.
	Suppose that
	\begin{align}\label{eq:h}
	\left(1 + z + z^2 + \cdots + z^{r}\right)^n = h_{n,0}(z^r) + z h_{n,1}(z^r) + \cdots + z^{r-1}h_{n,r-1}(z^{r}).
	\end{align}
	Then the polynomial sequence  $\left(h_{n,r-1}(z),\ldots, h_{n,1}(z), h_{n,0}(z)\right)$ is interlacing.
	In particular, $h \left(\Delta \left( \esd_r \left( 2^{[n]}\right) \right), z  \right)= h_{n,0}(z)$ has only real roots.
\end{cor}


\section*{\bf Acknowledgments}
This work was done during Philip  Zhang's visit to the University of Pennsylvania. 
The authors  would like to thank  Petter~Br{\"a}nd{\'e}n, Zhicong~Lin, and Vasu~Tewari for their helpful discussions.
James Haglund is  supported by NSF Grant DMS-1600670.
Philip Zhang is  supported by the National Science Foundation of China (No. 11701424).


\end{document}